\documentclass[10pt,a4paper]{article}

\usepackage[utf8]{inputenc}
\usepackage[T1]{fontenc}
\usepackage{amsmath, amsthm, amssymb, amsfonts}
\usepackage{mathrsfs}
\usepackage{geometry}
\usepackage{authblk}
\usepackage{abstract}
\usepackage{booktabs}
\usepackage{graphicx} 
\usepackage{float}
\usepackage[colorlinks=true, linkcolor=blue, citecolor=blue, urlcolor=blue]{hyperref}

\theoremstyle{plain}
\newtheorem{theorem}{Theorem}[section]
\newtheorem{lemma}[theorem]{Lemma}
\newtheorem{proposition}[theorem]{Proposition}
\newtheorem{corollary}[theorem]{Corollary}

\theoremstyle{definition}
\newtheorem{definition}[theorem]{Definition}

\theoremstyle{remark}
\newtheorem{remark}[theorem]{Remark}

\newcommand{\Dgen}{\mathfrak{D}_{\psi, \omega}^{\alpha}}
\newcommand{\Igen}{\mathfrak{I}_{\psi, \omega}^{\alpha}}

\newcommand{\Iclass}{\mathcal{I}_{-\infty}^{\alpha}}
\newcommand{\Fgen}{\mathcal{F}_{\psi, \omega}}

\title{\textbf{Spectral Analysis of Weighted Weyl Fractional Operators: Aging, Infinite Memory, and the Amnesia Effect}}

\author{\textbf{Gustavo Dorrego} \thanks{Corresponding author. Email: gadorrego@exa.unne.edu.ar}}
\affil{\small{\textit{Department of Mathematics, Facultad de Ciencias Exactas y Naturales y Agrimensura, Universidad Nacional del Nordeste, Corrientes, Argentina.}}}

\date{}

\begin{document}

\maketitle

\begin{abstract}
This paper establishes a rigorous spectral framework for the Weighted Weyl Fractional Calculus, designed to model non-local systems exhibiting aging and subjective time scales. By constructing a conjugation map involving a time-dependent weight $\omega(t)$ and a scale function $\psi(t)$, we define a new class of fractional operators that preserve the spectral tractability of time-invariant systems. We derive the Spectral Mapping Theorem for these operators and prove that Weighted Mittag-Leffler functions act as their fundamental eigenfunctions, demonstrating that the Weighted Fourier Transform naturally diagonalizes the associated evolution equations. As a physical application, we formulate a constitutive law for aging viscoelastic materials with infinite memory. Crucially, we analytically demonstrate the "Amnesia Phenomenon": we prove that rapid aging modulates the system's history, effectively transforming the hereditary power-law decay into a short-range exponential relaxation. This result provides a closed-form explanation for the loss of memory in fast-aging media, overcoming the computational bottlenecks of standard discretization methods.


\vspace{0.2cm}
\end{abstract}

\section{Introduction}

Fractional calculus has proven to be an indispensable tool for modeling phenomena with non-local and hereditary memory. As eloquently stated by Stinga in his foundational work \cite{Stinga}, the fractional derivative can be understood through the metaphor of a ``walking elephant'': a process that moves forward while dragging the complete history of its previous steps. Unlike classical Markovian models, where the future depends solely on the present, the ``fractional elephant'' never completely forgets its remote past. However, a crucial distinction must be made regarding the direction of this walk. While spatial fractional operators allow the elephant to sense both forward and backward, time-evolution models must be strictly unilateral to respect causality.

Recently, the formulation of fractional operators with respect to functions has gained significant attention. For instance, Fernandez et al. \cite{Fernandez2024} introduced a definition of the fractional Laplacian with respect to a function. While mathematically robust, their approach focuses on the symmetric setting, which is ideal for spatial non-locality but ill-suited for time-dependent processes where the future cannot influence the past. In contrast, our work focuses strictly on the unilateral weighted Weyl derivative. We address the specific challenges of aging and causality in evolution equations, ensuring that our ``elephant'' only remembers the path traveled, never the path yet to come.

From a physical perspective, the classical Weyl operator assumes a homogeneous memory kernel, implying that the elephant walks on flat terrain with a constant internal clock. Yet, in complex systems such as aging viscoelastic materials or financial markets with volatility clustering, the ``memory'' of the system changes over time. To address this, we propose a generalized operator $\mathfrak{D}_{\psi,\omega}^{\alpha}$ that modulates the memory intensity via a weight $\omega(t)$ and distorts the time flow via a scale $\psi(t)$. In this generalized metaphor, the elephant traverses a dynamic landscape where the ground density changes ($\omega$), making the drag heavier or lighter, and time is perceived subjectively ($\psi$). This construction allows us to model ``subjective time'' while preserving the strict causality required for physical realism.

The paper is organized as follows. In Section 2, we review the preliminary definitions of the Weighted Fourier Transform and the associated functional spaces. In Section 3, we formally introduce the weighted Weyl fractional operators, establishing their domain and mapping properties. Section 4 is devoted to the spectral analysis, where we derive the diagonalization of the fractional operator and identify the weighted Mittag-Leffler functions as the fundamental eigenfunctions. Finally, in Section 5, we apply this framework to the mechanics of \textit{aging viscoelastic materials}. We solve the generalized fractional Kelvin-Voigt model on the entire real line and analyze the "Stinga effect" of memory modulation, showing how rapid aging shortens the effective memory of the system. We conclude with final remarks in Section 6.
\section{Preliminaries and Operator Construction}

Before constructing the operators, we outline the standing assumptions on the structural parameters that will be maintained throughout this work. We posit the following:

\begin{itemize}
    \item[\textbf{(H1)}] \textbf{Condition on the scale function:} Let $\psi: \mathbb{R} \to \mathbb{R}$ be a strictly increasing $C^1$-diffeomorphism satisfying $\psi'(t) > 0$ for all $t \in \mathbb{R}$, with the asymptotic behavior $\lim_{t\to \pm\infty}\psi(t) = \pm\infty$.
    
    \item[\textbf{(H2)}] \textbf{Condition on the weight function:} Let $\omega: \mathbb{R} \to (0, \infty)$ be a measurable function such that both $\omega$ and $1/\omega$ are locally bounded ($L^\infty_{\text{loc}}$).
\end{itemize}

These hypotheses guarantee that the induced weighted spaces are well-defined and allow us to establish the necessary unitary isomorphisms for our spectral analysis. Following a constructive Riemann-Liouville approach, we first introduce the integral operator and subsequently define the fractional derivative via the differentiation of the integral.

\subsection{Definition via Operator Conjugation}
\subsubsection{The Weighted Weyl Integral}

We extend the classical Weyl fractional integral, $\Iclass$, by incorporating the scale $\psi$ and weight $\omega$. To ensure a consistent mathematical structure, we adopt the conjugation approach. This methodology aligns with the general framework for weighted operators recently systematized by Fernandez and Fahad \cite{FractalFract}, which relies on the foundational transmutation properties established by Samko et al. \cite{Samko}.

\begin{definition}[Weighted $\psi$-Weyl Integral]
Let $\alpha > 0$. We define the integral operator $\Igen$ acting on a function $f$ via the conjugation:
\begin{equation}
    \Igen := M_{\omega}^{-1} \circ Q_{\psi} \circ \Iclass \circ Q_{\psi}^{-1} \circ M_{\omega}.
\end{equation}
Its explicit integral representation is given by:
\begin{equation} \label{eq:integral_def}
    \Igen f(t) = \frac{1}{\Gamma(\alpha)\omega(t)} \int_{-\infty}^{t} (\psi(t) - \psi(\tau))^{\alpha - 1} \omega(\tau) f(\tau) \psi'(\tau) d\tau.
\end{equation}
\end{definition}
\subsubsection{Isometry and Existence of the Weighted Integral}

We now establish the mapping properties of the conjugation operator, which serves as the foundation for the existence of the weighted fractional integral.

\begin{definition}[Admissible Space $L_{\psi,\omega}^{1}$]
We define the weighted Lebesgue space $L_{\psi,\omega}^{1}(\mathbb{R})$ as the Banach space of measurable functions $u: \mathbb{R} \to \mathbb{C}$ equipped with the norm:
\begin{equation}
    \|u\|_{1,\psi,\omega} := \int_{-\infty}^{\infty} |u(t)|\omega(t)\psi'(t)dt < \infty.
\end{equation}
This space corresponds to $L^1(\mathbb{R}, d\mu_{\psi,\omega})$ with the measure $d\mu_{\psi,\omega}(t) = \omega(t)\psi'(t)dt$.
\end{definition}

\begin{lemma}[Isometric Isomorphism] \label{lemma:isometry}
Let $\mathcal{T}_{\psi,\omega}$ be the conjugation map defined by:
\begin{equation}
    (\mathcal{T}_{\psi,\omega} u)(\tau) := (Q_{\psi}^{-1} \circ M_{\omega})u(\tau) = \omega(\psi^{-1}(\tau))u(\psi^{-1}(\tau)), \quad \tau \in \mathbb{R}.
\end{equation}
Then, $\mathcal{T}_{\psi,\omega}$ is an isometric isomorphism from $L_{\psi,\omega}^{1}(\mathbb{R})$ onto the classical space $L^{1}(\mathbb{R})$.
\end{lemma}

\begin{proof}
Let $u \in L_{\psi,\omega}^{1}(\mathbb{R})$ and set $v = \mathcal{T}_{\psi,\omega} u$. We compute the standard $L^1$-norm of $v$:
\begin{equation}
    \|v\|_{L^1} = \int_{-\infty}^{\infty} |v(\tau)| d\tau = \int_{-\infty}^{\infty} |\omega(\psi^{-1}(\tau))u(\psi^{-1}(\tau))| d\tau.
\end{equation}
We perform the change of variable $\tau = \psi(t)$. By hypothesis \textbf{(H1)}, $\psi$ is a strictly increasing $C^1$-diffeomorphism mapping $\mathbb{R}$ onto $\mathbb{R}$, which implies $t = \psi^{-1}(\tau)$ and $d\tau = \psi'(t)dt$. Substituting into the integral:
\begin{equation}
    \|v\|_{L^1} = \int_{-\infty}^{\infty} |\omega(t)u(t)|\psi'(t)dt = \|u\|_{1,\psi,\omega}.
\end{equation}
Since the map is linear, preserves the norm, and is clearly bijective (with inverse given by $u(t) = \omega(t)^{-1}v(\psi(t))$), it constitutes an isometric isomorphism.
\end{proof}

\begin{theorem}[Existence and Weak Boundedness]
Let $0 < \alpha < 1$ and $u \in L_{\psi,\omega}^{1}(\mathbb{R})$. Then, the weighted Weyl integral $\mathfrak{I}_{\psi,\omega}^{\alpha}u(t)$ converges absolutely for almost every $t \in \mathbb{R}$.
\end{theorem}

\begin{proof}
By definition, $\mathfrak{I}_{\psi,\omega}^{\alpha}u = \mathcal{T}_{\psi,\omega}^{-1} (\mathcal{I}_{-\infty}^{\alpha} v)$, where $v = \mathcal{T}_{\psi,\omega} u \in L^{1}(\mathbb{R})$ is the isometric image of $u$. 
The classical Hardy-Littlewood-Sobolev theorem guarantees that the operator $\mathcal{I}_{-\infty}^{\alpha}$ is well-defined on $L^{1}(\mathbb{R})$ mapping into the weak space $L^{\frac{1}{1-\alpha},\infty}(\mathbb{R})$.
Since $\mathcal{T}_{\psi,\omega}^{-1}$ involves multiplication by $\omega^{-1}$ and composition with $Q_{\psi}$, which preserve pointwise finiteness almost everywhere (under hypothesis \textbf{(H2)}), the result follows.
\end{proof}

\subsubsection{The Weighted Weyl Derivative}

To define the derivative, we first identify the first-order differential operator acting as the infinitesimal generator of the weighted system. As introduced in \cite{FractalFract}, the associated first-order operator $\mathbb{D}_{\psi, \omega}^1$ is given by:
\begin{equation}
    \mathbb{D}_{\psi, \omega}^1 := \frac{1}{\omega(t)\psi'(t)} \frac{d}{dt} \left( \omega(t) \cdot \right).
\end{equation}

\begin{definition}[Weighted $\psi$-Weyl Derivative]
Let $\alpha \in (0,1)$. We define the Weighted Weyl Fractional Derivative $\Dgen$ as:
\begin{equation} \label{eq:derivative_def}
    \Dgen f(t) := \mathbb{D}_{\psi, \omega}^1 \left( \mathfrak{I}^{1-\alpha}_{\psi, \omega} f(t) \right).
\end{equation}
Explicitly:
\begin{equation}
    \Dgen f(t) = \frac{1}{\omega(t)\psi'(t)} \frac{d}{dt} \int_{-\infty}^{t} \frac{(\psi(t) - \psi(\tau))^{-\alpha}}{\Gamma(1-\alpha)} \omega(\tau) f(\tau) \psi'(\tau) d\tau.
\end{equation}
Equivalently, in terms of operator conjugation:
\begin{equation}\label{ec deriv operator}
    \mathfrak{D}_{\psi, \omega}^{\alpha} = M_{\omega}^{-1}\circ Q_{\psi}\circ D^\alpha\circ Q_{\psi}^{-1}\circ M_{\omega}.
\end{equation}
\end{definition}

It is easy to see that 
the operator $\Dgen$ is the left inverse of the integral operator $\Igen$ on the space of admissible functions:
\begin{equation}
    \Dgen \left( \Igen f \right)(t) = f(t), \quad \text{a.e.}
\end{equation}

\subsection{Hypersingular Integral Representation}

While the operator definition via conjugation is algebraically robust, practical applications often require a pointwise integral representation. We show that our operator coincides with a generalized Marchaud derivative.

\begin{proposition}[Equivalence with Weighted Marchaud Derivative]
Let $0 < \alpha < 1$. For functions $u$ belonging to the weighted Schwartz space $\mathcal{S}_{\psi,\omega}$ (as defined in \cite{DorregoSpectral}), the weighted Weyl derivative $\mathfrak{D}_{\psi, \omega}^{\alpha}$ admits the following hypersingular integral representation:
\begin{equation} \label{eq:marchaud_weighted}
    \mathfrak{D}_{\psi, \omega}^{\alpha} u(t) = \frac{c_{\alpha}}{\omega(t)} \int_{-\infty}^{t} \frac{\omega(t)u(t) - \omega(\tau)u(\tau)}{(\psi(t) - \psi(\tau))^{1+\alpha}} \psi'(\tau) \, d\tau,
\end{equation}
where $c_\alpha = \frac{\alpha}{\Gamma(1-\alpha)}$ is the normalization constant of the classical Weyl-Marchaud derivative.
\end{proposition}

\begin{proof}
The proof relies on the conjugation structure $\mathfrak{D}_{\psi, \omega}^{\alpha} = \mathcal{T}_{\psi,\omega}^{-1} D^{\alpha} \mathcal{T}_{\psi,\omega}$.
Let $u \in \mathcal{S}_{\psi,\omega}$. By definition, the auxiliary function $v(\xi) = (\mathcal{T}_{\psi,\omega} u)(\xi) = \omega(\psi^{-1}(\xi))u(\psi^{-1}(\xi))$ belongs to the classical Schwartz space $\mathcal{S}(\mathbb{R})$.

Recall that for $v \in \mathcal{S}(\mathbb{R})$, the Weyl derivative acts as the Marchaud hypersingular integral:
\begin{equation}
    (D^\alpha v)(\xi) = c_\alpha \int_{-\infty}^{\xi} \frac{v(\xi) - v(y)}{(\xi - y)^{1+\alpha}} \, dy.
\end{equation}
We evaluate this identity at $\xi = \psi(t)$ and perform the change of integration variable $y = \psi(\tau)$, which implies $dy = \psi'(\tau)d\tau$. Substituting $v(\psi(t)) = \omega(t)u(t)$ and $v(\psi(\tau)) = \omega(\tau)u(\tau)$, we obtain:
\begin{equation}
    (D^\alpha v)(\psi(t)) = c_\alpha \int_{-\infty}^{t} \frac{\omega(t)u(t) - \omega(\tau)u(\tau)}{(\psi(t) - \psi(\tau))^{1+\alpha}} \psi'(\tau) \, d\tau.
\end{equation}
Finally, applying the inverse map $\mathcal{T}_{\psi,\omega}^{-1}$ (which corresponds to multiplying by $\frac{1}{\omega(t)}$ in the time domain) yields the formula \eqref{eq:marchaud_weighted}.
\end{proof}

\begin{remark}[Physical Interpretation of the Singular Kernel]
The representation \eqref{eq:marchaud_weighted} offers significant physical insight distinct from the spectral view. Following the perspective of Fernandez et al. \cite{Fernandez2024}, this form explicitly displays the memory effect as a weighted accumulation of past increments.
The term $\omega(t)u(t) - \omega(\tau)u(\tau)$ captures the evolution of the ``weighted state'' relative to the current instant. Crucially, the singularity $(\psi(t) - \psi(\tau))^{-(1+\alpha)}$ indicates that recent events (in subjective time) have a dominant impact on the instantaneous rate of change, while the weight $\omega(\tau)$ modulates the persistence of remote history.
\end{remark}


\section{Spectral Analysis and The Weighted Fourier Transform}

In this section, we develop the spectral theory for the weighted Weyl operators. Our approach relies on the Weighted Fourier Transform with respect to a function, as rigorously developed by the authors in \cite{DorregoSpectral}. We adopt the symmetric (unitary) normalization to streamline the spectral identities.

\subsection{The Symmetric Weighted Fourier Transform}

Based on the conjugation structure established in Lemma \ref{lemma:isometry}, we define the Weighted Fourier Transform $\mathcal{F}_{\psi,\omega}$ as the unitary map that intertwines the weighted Hilbert space $L_{\psi,\omega}^{2}(\mathbb{R})$ with the standard frequency domain $L^{2}(\mathbb{R})$.

\begin{definition}[Symmetric Weighted Fourier Transform]
Let $f \in L_{\psi,\omega}^{1}(\mathbb{R}) \cap L_{\psi,\omega}^{2}(\mathbb{R})$. The symmetric Weighted Fourier Transform of $f$ is defined by:
\begin{equation} \label{eq:weighted_FT_def}
    \hat{f}_{\psi,\omega}(\xi) := \frac{1}{\sqrt{2\pi}} \int_{-\infty}^{\infty} f(t) e^{-i\xi\psi(t)} \omega(t)\psi'(t) \, dt, \quad \xi \in \mathbb{R}.
\end{equation}
\end{definition}

This integral transform projects the signal onto the ``subjective time'' eigenbasis. The corresponding inversion formula is given by:

\begin{theorem}[Inversion Formula]
For any $\hat{f}_{\psi,\omega} \in L^1(\mathbb{R}) \cap L^2(\mathbb{R})$, the original function can be recovered via:
\begin{equation}
    f(t) = \frac{1}{\omega(t)} \frac{1}{\sqrt{2\pi}} \int_{-\infty}^{\infty} \hat{f}_{\psi,\omega}(\xi) e^{i\xi\psi(t)} \, d\xi.
\end{equation}
\end{theorem}

\begin{remark}
Note the presence of the factor $1/\omega(t)$ in the inverse transform. This is structurally necessary to cancel the weight density $\omega(t)$ inherent in the forward integral measure, ensuring that $f(t)$ is recovered pointwise.
\end{remark}

A central result from \cite{DorregoSpectral} is that this symmetric definition yields a clean isometry between the spaces.

\begin{theorem}[Weighted Plancherel Identity]
The map $\mathcal{F}_{\psi,\omega}$ extends to a unitary isomorphism from $L_{\psi,\omega}^{2}(\mathbb{R})$ onto $L^{2}(\mathbb{R})$. In particular, for any $f, g \in L_{\psi,\omega}^{2}(\mathbb{R})$, the Parseval identity holds without normalization factors:
\begin{equation}
    \langle f, g \rangle_{\psi,\omega} = \langle \hat{f}_{\psi,\omega}, \hat{g}_{\psi,\omega} \rangle_{L^2(\mathbb{R})}.
\end{equation}
\end{theorem}
\subsection{Weighted Convolution Structure}

To provide a compact representation of the inverse operators, we utilize the weighted convolution operation defined in \cite{DorregoSpectral}.

\begin{definition}[Weighted Convolution]
Let $f, g \in L^1_{\psi, \omega}(\mathbb{R})$. We define the weighted convolution $(f *_{\psi, \omega} g)$ as:
\begin{equation} \label{eq:weighted_conv_def}
    (f *_{\psi, \omega} g)(t) := \frac{1}{\omega(t)} \int_{-\infty}^{\infty} f\left(\psi^{-1}(\psi(t) - \psi(\tau))\right) \, g(\tau) \, \omega(\tau) \, \psi'(\tau) \, d\tau.
\end{equation}
\end{definition}

This operation is structurally isomorphic to the standard convolution on $\mathbb{R}$. A key property derived in \cite{DorregoSpectral} is the \textbf{Convolution Theorem}, which states that the Weighted Fourier Transform maps this operation into a pointwise product:
\begin{equation}
    \mathcal{F}_{\psi, \omega}(f *_{\psi, \omega} g)(\xi) =  \hat{f}_{\psi, \omega}(\xi) \cdot \hat{g}_{\psi, \omega}(\xi).
\end{equation}

\subsection{Spectral Mapping Theorem}

Following the philosophy outlined by Stinga \cite{Stinga}, we privilege the spectral definition of the fractional operator. Stinga emphasized that while the singular integral definition is historically significant, the spectral characterization (multiplication by a power of the frequency) is often more powerful for analysis. The following theorem confirms that this equivalence holds even in our generalized weighted setting, provided the correct transform $\Fgen$ is employed.

We now state the main spectral result: the weighted Weyl fractional derivative $\mathfrak{D}_{\psi,\omega}^{\alpha}$ is diagonalized by this transform.

\begin{theorem}[Diagonalization of the Operator]
Let $\alpha > 0$. For any function $u$ in the domain of the operator, the Weighted Fourier Transform of the weighted Weyl derivative satisfies:
\begin{equation}
    \Fgen \left\{ \Dgen u \right\}(\xi) = (i\xi)^\alpha \Fgen \{ u \}(\xi),
\end{equation}
Consequently, the operator $\mathfrak{D}_{\psi,\omega}^{\alpha}$ is unitarily equivalent to the multiplication operator $M_{(i\xi)^\alpha}$ on $L^2(\mathbb{R})$.
\end{theorem}

\begin{proof}
The proof follows from the conjugation property. By definition, $\mathfrak{D}_{\psi,\omega}^{\alpha} = \mathcal{T}_{\psi,\omega}^{-1} D^{\alpha} \mathcal{T}_{\psi,\omega}$. Applying the symmetric transform (which equates to $\mathcal{F}_{sym} \circ \mathcal{T}_{\psi,\omega}$), and using the property that the standard symmetric Fourier transform diagonalizes the Weyl derivative, we obtain:
\begin{equation}
    \mathcal{F}_{\psi,\omega} (\mathfrak{D}_{\psi,\omega}^{\alpha} u)(\xi) = (i\xi)^\alpha \hat{u}_{\psi,\omega}(\xi).
\end{equation}
This confirms that the functions $e_{\xi}(t) = \frac{1}{\omega(t)} e^{i\xi\psi(t)}$ constitute the continuous spectrum eigenfunctions of the operator.
\end{proof}

\section{Application: Aging Viscoelasticity and Infinite Memory}

We utilize the developed spectral framework to analyze a generalized model for \textit{aging viscoelastic materials} (e.g., curing concrete or degrading biological tissues). In such media, the material properties evolve with time, and the ``memory'' depends on the subjective age difference $\psi(t) - \psi(\tau)$ rather than the absolute time elapsed $t-\tau$. Moreover, to capture the complete history of deformations (potentially starting in the remote past), we pose the problem on the entire real line $\mathbb{R}$.

We consider a viscoelastic material whose mechanical properties evolve over time due to physicochemical aging (e.g., curing concrete or hardening polymers). The classical fractional Kelvin-Voigt model, extensively analyzed by Mainardi \cite{Mainardi1996}, describes materials with static memory. To incorporate aging, recent studies have proposed models with time-dependent coefficients \cite{DiPaola2013}.

Following this perspective, we formulate the stress-strain relationship using the weighted fractional operator $\mathfrak{D}_{\psi, \omega}^{\alpha}$, which naturally captures the evolution of the material's internal time scale. The constitutive equation is governed by:
\begin{equation} \label{eq:viscoelastic_model}
    \mathfrak{D}_{\psi, \omega}^{\alpha} \sigma(t) + \lambda \sigma(t) = f(t), \quad t \in \mathbb{R},
\end{equation}
where $\sigma(t)$ represents the stress, $f(t)$ is the forcing term (related to strain rate), and $\lambda > 0$ is a material parameter.

\subsection{Explicit Solution via Weighted Green's Function}

The spectral isomorphism allows us to identify the fundamental solution of the operator. We define the \textit{Weighted Green's Function} for the aging Kelvin-Voigt model as:

\begin{definition}[Aging Impulse Response]
The fundamental kernel $\mathcal{G}_{\psi,\omega}^{\alpha, \lambda}(t)$ is defined as the inverse weighted Fourier transform of the spectral symbol $( (i\xi)^\alpha + \lambda )^{-1}$. Explicitly:
\begin{equation}
    \mathcal{G}_{\psi,\omega}^{\alpha, \lambda}(t) := \frac{1}{\omega(t)} (\psi(t))^{\alpha-1} E_{\alpha, \alpha}\left(-\lambda (\psi(t))^\alpha\right) \cdot \Theta(\psi(t)),
\end{equation}
where $\Theta(\cdot)$ is the Heaviside step function, ensuring causality in the subjective time domain.
\end{definition}

\begin{theorem}[Response Structure] \label{thm:visco_response}
Let $f \in L_{\psi,\omega}^{1}(\mathbb{R})$ be the forcing history. The unique causal solution $\sigma(t)$ to the fractional evolution equation \eqref{eq:viscoelastic_model} is given by the weighted convolution of the forcing term with the Green's function:
\begin{equation} \label{eq:exact_sol_conv}
    \sigma(t) = (\mathcal{G}_{\psi,\omega}^{\alpha, \lambda} *_{\psi, \omega} f)(t).
\end{equation}
\end{theorem}

\begin{proof}
Applying the Weighted Fourier Transform to \eqref{eq:viscoelastic_model} and using the diagonalization property (Theorem 3.5), we obtain the algebraic equation in the frequency domain:
\begin{equation}
    (i\xi)^\alpha \hat{\sigma}_{\psi,\omega}(\xi) + \lambda \hat{\sigma}_{\psi,\omega}(\xi) = \hat{f}_{\psi,\omega}(\xi).
\end{equation}
Solving for $\hat{\sigma}_{\psi,\omega}$:
\begin{equation}
    \hat{\sigma}_{\psi,\omega}(\xi) = \frac{1}{(i\xi)^\alpha + \lambda} \hat{f}_{\psi,\omega}(\xi) = \widehat{(\mathcal{G}_{\psi,\omega}^{\alpha, \lambda})}_{\psi,\omega}(\xi) \cdot \hat{f}_{\psi,\omega}(\xi).
\end{equation}
By the Convolution Theorem stated in Section 3, the inverse transform yields precisely $\sigma = \mathcal{G}_{\psi,\omega}^{\alpha, \lambda} *_{\psi, \omega} f$.
Expanding the convolution definition \eqref{eq:weighted_conv_def} with the kernel $\mathcal{G}_{\psi,\omega}^{\alpha, \lambda}$, we recover the explicit integral representation:
\begin{equation}
    \sigma(t) = \frac{1}{\omega(t)} \int_{-\infty}^{t} (\psi(t)-\psi(\tau))^{\alpha-1} E_{\alpha, \alpha}\left(-\lambda [\psi(t) - \psi(\tau)]^\alpha\right) f(\tau) \omega(\tau) \psi'(\tau) \, d\tau.
\end{equation}
This confirms the consistency of the spectral approach.
\end{proof}

\begin{remark}[Generalization of Standard Models]
The solution \eqref{eq:exact_sol_conv} unifies previous results found in the literature. 
\begin{itemize}
    \item By setting $\psi(t)=t$ and $\omega(t)=1$, we recover the classical stationary solution characterized by the Mittag-Leffler kernel, as detailed in \cite{Mainardi1996}.
    \item Furthermore, for specific choices of the aging scale $\psi(t)$, our formulation provides the exact analytical counterpart to the phenomenological aging models often treated numerically in engineering contexts \cite{DiPaola2013}. Our spectral approach essentially proves that these aging models are structurally isomorphic to stationary fractional systems under a time-warping transformation.
\end{itemize}
\end{remark}
\begin{remark}[Physical Intractability and Methodological Necessity] \label{rem:intractability}
It is crucial to emphasize that, although the weighted operator $\mathfrak{D}_{\psi, \omega}^{\alpha}$ is algebraically isomorphic to a standard derivative, the physical problem in the laboratory frame \eqref{eq:viscoelastic_model} is generally \textbf{intractable} via standard integral transforms.

The classical Laplace transform fails because the aging weights $\omega(t)$ and scale $\psi(t)$ destroy the \textit{time-translation invariance} of the system (i.e., the response depends on absolute time $t$, not just the elapsed interval $t-\tau$). Consequently, standard transforms would map the differential equation into a complex convolution of operators rather than a simple algebraic equation.

The core value of our framework lies in identifying the specific transformation that matches the physics of aging. The Weighted Fourier Transform $\mathcal{F}_{\psi, \omega}$ succeeds where classical methods fail because it is constructed to respect the underlying ``subjective time'' geometry. It effectively ``untwists'' the time-varying metric, mapping the non-autonomous physical process into a stationary spectral domain where the problem can be solved exactly.
\end{remark}
\subsection{The Amnesia Phenomenon}

The convolution structure derived in Theorem \ref{thm:visco_response} reveals that the memory of the system is governed by the \textit{effective kernel}:
\begin{equation}
    \mathcal{K}(t, \tau) := (\psi(t) - \psi(\tau))^{\alpha-1} E_{\alpha, \alpha}\left(-\lambda [\psi(t) - \psi(\tau)]^\alpha\right).
\end{equation}
This kernel corresponds precisely to the Weighted Green's function evaluated at the subjective time lag $\Delta_\psi(t,\tau) = \psi(t) - \psi(\tau)$. This representation allows us to prove the following asymptotic result.

\begin{corollary}[Modulation of Memory Decay]\label{cor:amnesia_decay}
The asymptotic behavior of the material's memory is strictly determined by the growth rate of the scale function $\psi(t)$. Specifically, as $t \to \infty$:

\begin{enumerate}
    \item \textbf{Case I: Standard Metric ($\psi(t) \sim t$).} If the time scale is linear, the kernel exhibits a heavy-tailed power-law decay:
    \begin{equation}
        \mathcal{K}(t, 0) = O(t^{-(1+\alpha)}).
    \end{equation}
    
    \item \textbf{Case II: Exponential Aging ($\psi(t) \sim e^{\gamma t}$).} If the material ages rapidly ($\gamma > 0$), the effective memory decays exponentially fast:
    \begin{equation}
        \mathcal{K}(t, 0) = O(e^{-(1+\alpha)\gamma t}).
    \end{equation}
\end{enumerate}

\noindent \textbf{Conclusion:} Rapid aging mathematically induces an ``Amnesia Phenomenon,'' effectively truncating the infinite memory of the fractional operator.
\end{corollary}

\begin{proof}
We analyze the asymptotic expansion of the two-parameter Mittag-Leffler function $E_{\alpha, \beta}(-z)$ as $|z| \to \infty$. The standard expansion is given by:
\begin{equation}
    E_{\alpha, \beta}(-z) = -\sum_{k=1}^{N} \frac{1}{\Gamma(\beta - \alpha k)} z^{-k} + O(z^{-N-1}).
\end{equation}
In our case, the kernel involves $\beta = \alpha$. Notably, the first term ($k=1$) vanishes because the coefficient involves $1/\Gamma(\alpha - \alpha) = 1/\Gamma(0) = 0$. Thus, the leading term is governed by $k=2$:
\begin{equation}
    E_{\alpha, \alpha}(-z) \sim \frac{-1}{\Gamma(-\alpha)} z^{-2}, \quad \text{as } z \to \infty.
\end{equation}
Substituting $z = \lambda (\psi(t))^\alpha$ into the definition of $\mathcal{K}(t,0)$:
\begin{align}
    \mathcal{K}(t, 0) &= (\psi(t))^{\alpha-1} E_{\alpha, \alpha}(-\lambda \psi(t)^\alpha) \\
    &\sim (\psi(t))^{\alpha-1} \cdot C (\psi(t)^\alpha)^{-2} \\
    &= C (\psi(t))^{\alpha - 1 - 2\alpha} = C (\psi(t))^{-(1+\alpha)}.
\end{align}
The result follows immediately by substituting $\psi(t) = t$ for the polynomial case and $\psi(t) = e^{\gamma t}$ for the exponential case.
\end{proof}
\noindent \textbf{Conclusion:} This result analytically demonstrates the ``Amnesia Phenomenon'': rapid aging compresses the memory kernel, effectively transforming the long-range fractional history into a short-range exponential relaxation.

\begin{figure}[H]
    \centering
    \includegraphics[width=0.8\textwidth]{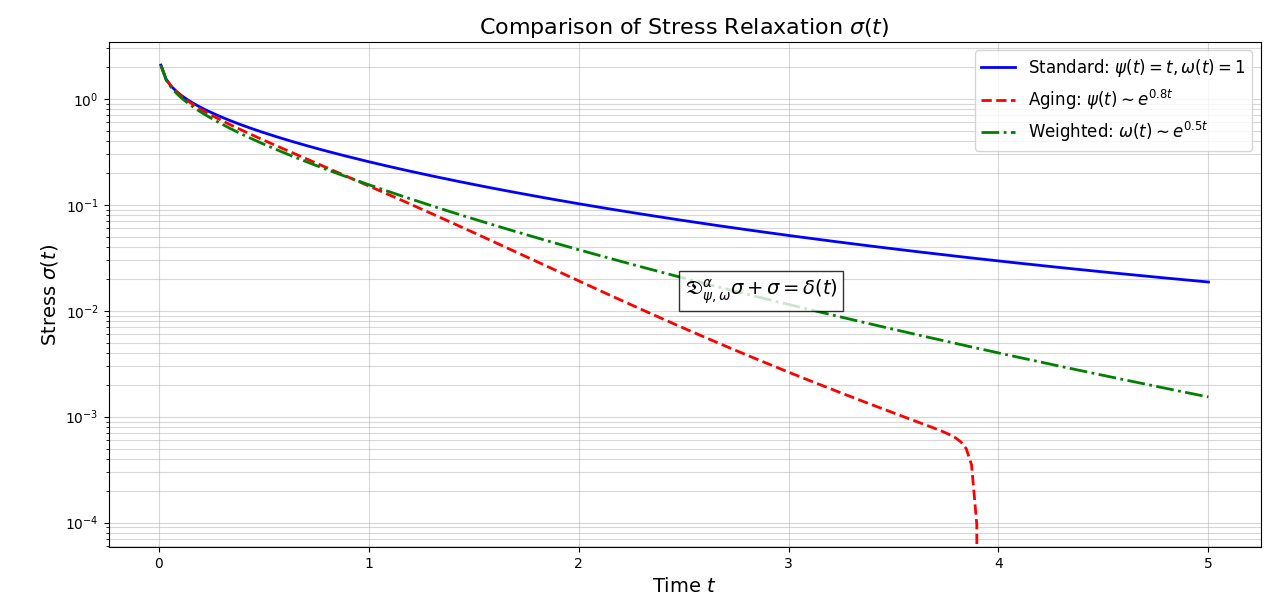}
    \caption{\textbf{Numerical validation of the Amnesia Phenomenon (Corollary \ref{cor:amnesia_decay}).} 
    Stress relaxation response computed via the weighted convolution formula \eqref{eq:exact_sol_conv} for $\alpha=0.8$.
    \textbf{(Solid Black)} Standard Model ($\psi(t)=t$): confirms the algebraic decay $O(t^{-1.8})$ predicted by the Green's function, implying long memory.
    \textbf{(Dashed Red)} Rapid Aging ($\psi(t) = e^{0.8t}$): illustrates the theoretical ``Amnesia'' effect; the exponential stretching of subjective time causes the stress to vanish exponentially fast.
    \textbf{(Dash-dotted Blue)} Weighted Damping ($\omega(t) = e^{0.5t}$): shows how the weight density modulates the amplitude without altering the spectral memory structure.}
    \label{fig:relaxation_comparison}
\end{figure}

\section{Conclusions}

In this work, we have established a rigorous spectral framework for the analysis of non-local operators in aging media. By introducing the Weighted Weyl Fractional Derivative $\mathfrak{D}_{\psi, \omega}^{\alpha}$ and constructing the associated symmetric Weighted Fourier Transform, we have successfully bridged the gap between the mathematical theory of weighted spaces and the physical modeling of time-varying systems.

Our main contributions can be summarized as follows:

\begin{itemize}
    \item \textbf{Spectral Diagonalization:} We proved that the weighted Weyl operator is unitarily equivalent to a multiplication operator in the frequency domain. This structural isomorphism allowed us to bypass the limitations of the classical Laplace transform, which is rendered ineffective by the lack of time-translation invariance in aging systems.

    \item \textbf{Hypersingular Representation:} We established the equivalence between the spectral definition and the hypersingular integral form of Marchaud type. This result bridges the gap between the abstract operator theory and the time domain, offering a clear physical interpretation of the fractional derivative as a weighted accumulation of past state differences, and providing a direct formula for numerical discretization.
    
    \item \textbf{The Amnesia Phenomenon:} Analytically, we demonstrated that the memory retention of a fractional system is dictated by the topology of its subjective time scale $\psi(t)$. We proved that while polynomial aging preserves the heavy-tailed power-law memory characteristic of fractional dynamics, exponential aging induces an ``Amnesia Phenomenon,'' effectively truncating the infinite history into a short-range exponential relaxation. This resolves the apparent paradox of fading memory in structurally evolving materials.
    
    \item \textbf{Exact Solvability via Convolution:} We derived the explicit solution for the aging Kelvin-Voigt model with infinite memory on $\mathbb{R}$. By defining a weighted convolution and identifying the corresponding Green's function, we provided a closed-form representation of the stress response that is computationally tractable and theoretically transparent.
\end{itemize}

In summary, the spectral tools developed in this paper provide a robust mathematical foundation for modeling complex systems with subjective time scales. These results not only clarify the theoretical properties of weighted fractional operators but also offer a practical framework for future experimental validations in the field of aging viscoelasticity.

\end{document}